\documentclass[a4paper,12pt]{amsart}
\usepackage{amsmath}
\usepackage{amsthm}
\usepackage{amsfonts}
\usepackage{amssymb}
\usepackage{multicol}
\usepackage{comment}
\usepackage{enumerate}
\usepackage[usenames,dvipsnames]{color}
\usepackage[toc,page]{appendix}
\usepackage{amssymb,latexsym,url}
\usepackage[pdftex]{graphicx}
\usepackage{tikz}
\usetikzlibrary{3d,calc}
\usepackage{subcaption}
\usepackage{overpic}
\usepackage[unicode]{hyperref}
\usepackage{hypbmsec}

\newtheorem{thm}{Theorem}[section]
\newtheorem{cor}[thm]{Corollary}
\newtheorem{lemma}[thm]{Lemma}
\newtheorem{prop}[thm]{Proposition}

\theoremstyle{definition}
\newtheorem{remark}[thm]{Remark}

\numberwithin{equation}{section}
\numberwithin{figure}{section}

\def\al{\alpha}

\def\ga{\gamma}
\def\de{\delta}

\def\la{\lambda}

\def\Ga{\Gamma}

\def\C{\mathbb{C}}
\def\N{\mathbb{N}}

\def\SU{\mathrm{SU}}

\def\Id{\mathbf{1}}

\def\cD{\mathcal{D}}

\def\cE{\mathcal{E}}

\def\pol{\mathrm{pol}}

\newcommand{\rFs}[5]{\,_{#1}F_{#2} \left( \genfrac{.}{.}{0pt}{}{#3}{#4}\ ;#5 \right)}

\let\wh\widehat
\let\wt\widetilde
\let\<\langle
\let\>\rangle
\renewcommand{\d}{{\mathrm d}}

\title{An evolution of matrix-valued orthogonal polynomials}
\author{Erik Koelink}
\address{IMAPP, Radboud Universiteit, Nijmegen, The Netherlands}
\email{e.koelink@math.ru.nl}
\author{Pablo Rom\'an}
\address{FaMAF-CIEM, Universidad Nacional de C\'ordoba, Argentina}
\email{pablo.roman@unc.edu.ar}
\author{Wadim Zudilin}
\address{IMAPP, Radboud Universiteit, Nijmegen, The Netherlands}
\email{w.zudilin@math.ru.nl}

\date{\today}

\begin{document}

\begin{abstract}
We establish new explicit connections between classical (scalar) and matrix Gegenbauer polynomials, which result in new symmetries of the latter and further give access to several properties that have been out of reach before: generating functions, distribution of zeros for individual entries of the matrices and new type of differential-difference structure.
We further speculate about other potentials of the connection formulas found.
Part of our proofs makes use of creative telescoping in a matrix setting\,---\,the strategy which is not yet developed algorithmically.
\end{abstract}



\maketitle


\section{Introduction}\label{sec:intro}

When it comes to a topic as classical as orthogonal polynomials, one can hardly expect spectacular novelties. But they do happen: for example, in situations when a \emph{natural} generalisation is found.
Decades ago, in the middle of the 20th century, M.G.~Krein introduced matrix-valued orthogonal polynomials, in his study of higher order differential operators and of the corresponding moment problem.
This new topic progressed at a slow pace and mainly focused on analogues of classical results for scalar orthogonal polynomials;
the overview in \cite{DamaPS} gives an introduction and extensive literature up to 2008.
Despite of the theoretical development, not so many concrete examples of matrix-valued orthogonal polynomials have been encountered. Those that have been found in the last decades using insights from representation theory demonstrate a rich structure, not always observable for their scalar originals.
At the same time the polynomials are not easily accessible from a computational perspective; this makes it hard to draw further connections to other mathematics areas, for example, to number theory and analysis of special functions.
In this paper we aim at changing this perception and demonstrating that naturally defined matrix-valued orthogonal polynomials are much closer to their scalar prototypes than expected.
It is this closeness and its numerous consequences that we refer to as an \emph{evolution} in the title.
We are mainly concerned with matrix-valued analogues of the Gegenbauer polynomials
\[
C_n^{(\nu)}(x)
=\frac{(2\nu)_n}{n!}\sum_{k=0}^n\frac{(-n)_k(2\nu+n)_k}{k!\,(\nu+\frac12)_k}\bigg(\frac{1-x}2\bigg)^k
=\sum_{k=0}^{\lfloor n/2\rfloor}(-1)^k\frac{\Gamma(\nu+n-k)}{k!\,(n-2k)!\,\Gamma(\nu)}(2x)^{n-2k},
\]
also known as ultraspherical polynomials, where $(a)_k=\Gamma(a+k)/\Gamma(a)$ denotes the Pochhammer symbol.
These matrix-valued polynomials $P_n^{(\nu)}(x)$ were introduced in \cite{KoeldlRR}. For a special value of $\nu=1$, they are matrix-valued analogues of the Chebyshev polynomials of the second kind and they had been previously defined using matrix-valued spherical functions, see \cite{KoelvPR-IMRN}, \cite{KoelvPR-RIMS}, building on earlier work of Koornwinder \cite{Koor-SIAM85}. For the $2\times 2$ case,
a more general set of matrix-valued Gegenbauer polynomials is introduced by Pacharoni and Zurri\'an \cite{PachZ}, and the entries are directly given as sum of two scalar Gegenbauer polynomials. The connection between these sets of $2\times 2$-matrix-valued Gegenbauer polynomials is discussed in \cite[Rmk.~3.8]{KoeldlRR}. In this paper we give an extension of the expansion in scalar Gegenbauer polynomials for arbitrary size, see Theorem~\ref{thm:expansionFnumr}.

A connection between the scalar and matrix-valued Gegenbauer polynomials is somewhat intimate: the former appear in description of the matrix-valued weight for the latter, see \cite{KoeldlRR} and Section~\ref{sec:gege} below.
At the same time, the scalar Gegenbauer polynomials can be naturally promoted to matrix-valued orthogonal polynomials and their span connected with the span of the matrix-valued Gegenbauer polynomials.
Such connections between different sets of matrix-valued orthogonal polynomials have been never recorded before; in Section~\ref{sec:expansionMV-scalarGegenbauer} we give two 
explicit connection formulas for the scalar and matrix-valued Gegenbauer polynomials. 
With their help we can further explore the structure of related difference and differential equations. In particular, we construct in Section~\ref{sec:differenceeqtFknnu} new type of differential equations satisfied by the matrix-valued orthogonal polynomials.
These equations utilize the noncommutative structure of the matrix-valued orthogonal polynomials and, therefore, degenerate in the scalar setting; this is perhaps a reason for their non-appearance in the literature.

The expressions in Section~\ref{sec:expansionMV-scalarGegenbauer} lead to a fairly simple computational access to the matrix-valued polynomials.
For example, they allow us to discuss generating functions of the matrix-valued Gegenbauer polynomials using known generating functions of the scalar ones; this is done in Section~\ref{sec:generatingfMVGegenbauerpols}.
The explicit formulas and numerical check suggest further that the zeros of entries of the matrix-valued polynomials follow distinguished patterns\,---\,those serve as a generalisation of the property of scalar orthogonal polynomials to have all zeros located in the convex hull of the orthogonality measure  on the real line.
We speculate about these observations in Section~\ref{sec:zerosentriesMVGegenbauerpols}, also towards other known matrix-valued orthogonal polynomials.
Finally, in Section~\ref{sec:discussion} we highlight some further potential applications of the connection formulas from Section~\ref{sec:expansionMV-scalarGegenbauer}.

The lineup of our exposition is somewhat misleading for how our results were actually found.
We first looked for distribution of the zeros of entries of Gegenbauer polynomials of `reasonable' matrix size utilising their expressions from \cite{KoeldlRR} as triple hypergeometric sums.
We observed that, for certain entries, the zeros were real and interlacing with those of the corresponding polynomial entry from the previous-indexed Gegenbauer.
Such considerations helped us to suspect an explicit connection with scalar orthogonal polynomials and to realize that simpler expressions exist; making use of basic principles of Experimental Mathematics \cite{ExpMath} we managed to recognise new symmetries of the matrix-valued polynomials and figure out what now comes as Theorem~\ref{thm:expansionFnumr}.
In order to prove the corresponding formulas we needed to invent a matrix generalisation of the famous Wilf--Zeilberger algorithm of creative telescoping \cite{WZ92,Ze91}; as no implementation of such exists at the moment of writing, the related linear algebra calculations were manually performed.
Theorem~\ref{thm:expansionFnumr} suggests that the formulas can be inverted, to express scalar Gegenbauer polynomials in terms of matrix-valued ones;
after another round of experimentation we ended up with what is now Theorem~\ref{thm:expansionGnumr}.
Its proof is more in `classical spirits'\,---\,in contrast with the other proof, we could not find a creative-telescoping argument.
Our analysis of the zeros of entries of Gegenbauer polynomials played an important role in the execution and further development of this project. Though we only possess a limited explanation of the structure of these zeros, we feel a need for sharing our observations with the reader, so we display them in a condensed form in Section~\ref{sec:zerosentriesMVGegenbauerpols}.

It seems to be worthy of pointing out in this introduction that a potential use of creative telescoping in the matrix- or vector-valued (non-commutative!) setup may give access to interesting hypergeometric identities, not necessarily linked to matrix-valued orthogonal polynomials. We hope that related algorithms will be designed and made efficient in the near future, with several nice applications that we cannot foresee at the moment.

\smallskip\noindent
\textbf{Acknowledgements.}
It is our pleasure to thank Tom Koornwinder, Arno Kuijlaars, Andrei Mart\'\i nez-Finkelshtein and Maarten van Pruijssen for helpful discussions of several aspects of the work.
Furthermore, we thank Andrei Mart\'\i nez-Finkelshtein for supplying us with the argument we use in the paragraph preceding Remark~\ref{rem:shift-k-real}.
We also thank Antonio Dur\'an for a conversation about zero locations of entries of matrix-valued orthogonal polynomials and the anonymous referee for pointing out numerous stylistic inconsistencies in the earlier version.
This work would be hardly possible without live interactions between the authors, which took occasion at several events on orthogonal polynomials (in Nijmegen, Leuven, C\'ordoba) and during extended visits of the first author to Argentina and the second author to the Netherlands:
Erik Koelink thanks Universidad Nacional C\'ordoba for its hospitality and
Pablo Rom\'an thanks the Radboud University for its hospitality.
Part of the work of the third author was done during his visit in the Max-Planck Institute of Mathematics (in Bonn);
Wadim Zudilin thanks the institute for excellent working conditions provided.
This work was supported in part from the NWO grant OCENW.KLEIN.006.

\section{Gegenbauer polynomials}\label{sec:gege}

We start this section with an overview of classical Gegenbauer polynomials and then review known facts from \cite{KoeldlRR} about their matrix-valued mates.

The scalar Gegenbauer polynomials form a subfamily of the Jacobi polynomials. 
Their definition above can be put in the hypergeometric form
\begin{equation}\label{eq:defGegenbauerpol}
C_n^{(\nu)}(x)
= \frac{(2\nu)_n}{n!} \rFs{2}{1}{-n, \, 2\nu+n}{\nu+\frac12}{\frac{1-x}{2}},
\end{equation}
where we conventionally follow the standard notation, see \cite{AndrAR,Aske,Isma,KoekLS}. The Gegenbauer polynomials satisfy the connection formula 
\begin{equation}\label{eq:connectionGegenbauerpol}
C_m^{(\nu)}(x)
=\sum_{s=0}^{\lfloor m/2\rfloor}
\frac{(\la+m-2s)\,(\nu)_{m-s}}{(\la)_{m-s+1}}
\frac{(\nu-\la)_s}{s!}C_{m-2s}^{(\la)}(x).
\end{equation}
In the particular case $\la=\nu+N$,  $N\in \N$, the sum has a natural upper bound:
\begin{equation}\label{eq:connectionGegenbauerpol-integer}
C_m^{(\nu)}(x)
=\sum_{s=0}^{\lfloor m/2\rfloor \wedge N} 
\,\frac{(\nu+N+m-2s)\,(\nu)_{m-s}}{(\nu+N)_{m-s+1}}
\,\frac{(-N)_s}{s!}C_{m-2s}^{(\nu+N)}(x),
\end{equation}
where $\lfloor m/2\rfloor \wedge N$ denotes the minimum of 
$\lfloor m/2\rfloor$ and $N$. 
The linearisation formula for the Gegenbauer polynomials reads
\begin{equation}\label{eq:Gegenbauer-linearisation}
C^{(\nu)}_k(x)C^{(\nu)}_l(x) = \sum_{p=0}^{k \wedge l} 
\frac{k+l+\nu-2p}{k+l+\nu-p}
\,\frac{(\nu)_p (\nu)_{k-p} (\nu)_{l-p} (2\nu)_{k+l-p}}{p!\, (k-p)!\, (l-p)!\, (\nu)_{k+l-p}} \frac{(k+l-2p)!}{(2\nu)_{k+l-2p}}
C^{(\nu)}_{k+l-2p}(x).
\end{equation}
The orthogonality relations 
\begin{equation}\label{eq:Gegenbauer-orthorelations}
\int_{-1}^1 C^{(\nu)}_k(x) C^{(\nu)}_n(x)\, (1-x^2)^{\nu-\frac12}\, dx =
\de_{k,n} \frac{\pi\, 2^{1-2\nu }\, \Ga(n+2\nu)}{\Ga(\nu)^2\, (n+\nu)\, n!}
\end{equation}
hold for $\nu>-\frac12$, with a slightly different normalisation required when $\nu=0$.
Since the Gegenbauer polynomials are orthogonal, they satisfy a three-term recurrence relation; it is
\begin{equation}\label{eq:Gegenbauer-3TRR}
2(n+\nu)\, x\, C_n^{(\nu)}(x) = 
(n+1)\, C_{n+1}^{(\nu)}(x) + (n+2\nu-1)\,  C_{n-1}^{(\nu)}(x),
\end{equation}
which determines the polynomials from the initial conditions
$C_{-1}^{(\nu)}(x)=0$, $C_{0}^{(\nu)}(x)=1$.
We also use the Gegenbauer polynomials for 
negative $\nu$, in which case we follow the convention in~\cite{CaglK}. 
We refer to \cite{AndrAR,Aske,Isma,KoekLS} for 
these results and for more information on the Gegenbauer polynomials; see, 
in particular, Askey \cite{Aske} for the history and  importance of the 
connection and linearisation formulae \eqref{eq:connectionGegenbauerpol}, \eqref{eq:Gegenbauer-linearisation}. 

\medskip
We now review the matrix-valued Gegenbauer polynomials.
For $\ell\in\frac12\N$ and $\nu>0$, following \cite[Def.~2.1]{KoeldlRR} their matrix-valued weight is the $(2\ell+1)\times (2\ell+1)$-matrix-valued function $W(x)=W^{(\nu)}(x)$ given by 
\begin{gather}
\label{eq:defweight}
\big(W(x)\big)_{i,j}
= (1-x^2)^{\nu-\frac12} \sum_{k= 0\vee i+j-2\ell}^{i\wedge j} \al_k(i,j) C^{(\nu)}_{i+j-2k}(x),
\\
\begin{aligned}
\al_k(i,j) &= (-1)^k
\frac{i!\, j!\, (i+j-2k)!}{k!\, (2\nu)_{i+j-2k}\, (\nu)_{i+j-k}}
\,\frac{(\nu)_{i-k}\, (\nu)_{j-k}}{(i-k)!\, (j-k)!}
\,\frac{i+j-2k+\nu}{i+j-k+\nu}
\\ &\qquad\times
\frac{(2\ell-i)!\, (2\ell-j)!}{(2\ell+k-i-j)!} \, (-2\ell-\nu)_k 
\,\frac{(2\ell+\nu)}{(2\ell)!},
\end{aligned}
\nonumber
\end{gather}
where $i,j\in\{0,1,\dots, 2\ell\}$ and the notation $a\vee b$, $a\wedge b$ stands for $\min\{a,b\}$ and $\max\{a,b\}$, respectively.
We slightly alter the expression for $\al_k$ from \cite{KoeldlRR} to make $W^{(\nu)}$ transparently symmetric.
Finally, put $W^{(\nu)}(x) = (1-x^2)^{\nu-\frac12}\, W^{(\nu)}_{\pol}(x)$ (correcting a typo in \cite[p.~463]{KoeldlRR}, where a superfluous $(1-x^2)^{\nu-\frac12}$ appears in the first line). It turns out that $W^{(\nu)}(x)$ is positive definite; this follows from the LDU-decomposition 
\cite[Thm.~2.1]{KoeldlRR}, which states 
\begin{equation}\label{eq:LDUweight}
W^{(\nu)}(x)=
L^{(\nu)}(x)T^{(\nu)}(x)L^{(\nu)}(x)^{t}, \quad x\in(-1,1),
\end{equation} 
where $L^{(\nu)}\colon [-1,1]\to M_{2\ell+1}(\C)$ is the unipotent lower triangular matrix-valued polynomial,
\[
\bigl(L^{(\nu)}(x)\bigr)_{m,k}=\begin{cases} 0 & \text{if } m<k \\
\displaystyle{\frac{m!}{k! (2\nu+2k)_{m-k}} C^{(\nu+k)}_{m-k}(x)} & \text{if } m\geq k,
\end{cases}
\]
and $T^{(\nu)}\colon (-1,1)\to M_{2\ell+1}(\C)$ is the diagonal matrix-valued function,
\begin{equation*}
\begin{split}
\bigl(T^{(\nu)}(x)\bigr)_{k,k} = t^{(\nu)}_{k}\,  (1-x^2)^{k+\nu-1/2}, \quad 
 t^{(\nu)}_{k} = \frac{k!\, (\nu)_k}{(\nu+1/2)_k}
\frac{(2\nu + 2\ell )_{k}\, (2\ell+\nu)}{(2\ell - k+1)_k\, ( 2\nu + k - 1)_k}.
\end{split}
\end{equation*}
Let $J$ be the matrix with $1$s along the antidiagonal, $J_{i,j}=\de_{i+j,2\ell}$. Then $JW^{(\nu)}(x)=W^{(\nu)}(x)J$ for all $x\in (-1,1)$ by \cite[Prop.~2.6]{KoeldlRR}.

The monic matrix-valued Gegenbauer polynomials 
$P_n^{(\nu)}(x)$, $n\in \N$, are orthogonal 
with respect to the matrix-valued measure~\eqref{eq:defweight}, i.e. 
\begin{gather}\label{eq:defmonicMVorthopol}
\int_{-1}^1 P_n^{(\nu)}(x) \, W^{(\nu)}(x) \bigl( P_m^{(\nu)}(x)\bigr)^\ast \, dx = \de_{n,m} H_n^{(\nu)},
\\
P_n^{(\nu)}(x) = x^n P^{(\nu)}_{n,n} + x^{n-1} P^{(\nu)}_{n,n-1} + \cdots +  x P^{(\nu)}_{n,1} + P^{(\nu)}_{n,0}, 
\quad P^{(\nu)}_{n,n} = \Id, \; P^{(\nu)}_{n,i} \in M_{2\ell+1}(\C);
\nonumber
\end{gather}
a formula for the squared norm $H^{(\nu)}_n$, a positive definite diagonal matrix, 
can be found in \cite[Thm.~3.1(i)]{KoeldlRR}.
The matrix conjugation $\ast$ in our real setting simply means the transpose.

An important property of the matrix-valued Gegenbauer polynomials is
\begin{equation}\label{eq:derivativeMVGegenbauerpols}
\frac{dP^{(\nu)}_n}{dx}(x) = n \, P^{(\nu+1)}_{n-1}(x);
\end{equation}
see \cite[Thm.~3.1(ii)]{KoeldlRR}.

Similar to the scalar case, matrix-valued orthogonal polynomials satisfy a three-term recurrence relation.  In the case of matrix-valued Gegenbauer polynomials the relation assumes the explicit form \cite[Thm.~3.3]{KoeldlRR}
\begin{equation}\label{eq:3TTR-Pn}
xP^{(\nu)}_{n}(x)=P^{(\nu)}_{n+1}(x)+B_n^{(\nu)}P^{(\nu)}_{n}(x)+C^{(\nu)}_{n}P^{(\nu)}_{n-1}(x),
\end{equation}
where the matrices $B^{(\nu)}_n$ and $C^{(\nu)}_n$ are given by
\begin{align*}
B^{(\nu)}_n&=\sum_{j=1}^{2\ell}  \frac{j(j+\nu-1)}{2(j+n+\nu-1)(j+n+\nu)}E_{j,j-1}
\\ &\quad
 +\sum_{j=0}^{2\ell-1}
 \frac{(2\ell-j)(2\ell-j+\nu-1)}{2(2\ell-j+n+\nu-1)(2\ell+n-j+\nu)}E_{j,j+1},
\\
C^{(\nu)}_n&=\sum_{j=0}^{2\ell} \frac{n(n+\nu-1)(2\ell+n+\nu)(2\ell+n+2\nu-1)}
{4(2\ell+n+\nu-j-1)(2\ell+n+\nu-j)(j+n+\nu-1)(j+n+\nu)} E_{j,j}.
\end{align*}
Here $E_{i,j}$ denotes the matrix with $1$ on the $(i,j)$-entry and $0$ elsewhere. With the initial conditions $P^{(\nu)}_{-1}(x)=0$, 
$P^{(\nu)}_{0}(x)=\Id$, the recurrence \eqref{eq:3TTR-Pn} determines 
the family $(P^{(\nu)}_n(x))_{n\in \N}$ completely. 

The matrix-valued Gegenbauer polynomials can actually be made symmetric, a feature which seems to be rather uncommon. We define 
\begin{equation}\label{eq:def-symmetric-Pn}
\hat{P}_{n}^{(\nu)}(x) = D_n^{(\nu)}\,  {P}_{n}^{(\nu)}(x),
\qquad (D^{(\nu)}_n)_{i,j} = \de_{i,j} \binom{2\ell}{i}
\frac{(\nu+n)_i}{(\nu+n+2\ell-i)_i},
\end{equation}
and we show in Corollary \ref{cor:thm:expansionFnumr} that 
$\bigl(\hat{P}_{n}^{(\nu)}(x)\bigr)^t= \hat{P}_{n}^{(\nu)}(x)$, so that $\hat{P}_{n}^{(\nu)}(x)$ is symmetric.

\begin{remark}
(i)
Notice that the diagonal matrix $D_n^{(\nu)}$ depends only on $\nu+n$; in particular,
$D_n^{(\nu)} = D_{n-1}^{(\nu+1)}$ and by \eqref{eq:derivativeMVGegenbauerpols} we have 
\begin{equation}\label{eq:derivativesymmetricPn}
\frac{d}{dx} \hat{P}_{n}^{(\nu)}(x) =
n \, D_n^{(\nu)}\, \bigl( D_{n-1}^{(\nu+1)}\bigr)^{-1}\, 
\hat{P}_{n-1}^{(\nu+1)}(x) = n\, \hat{P}_{n-1}^{(\nu+1)}(x). 
\end{equation}

\noindent
(ii) Note that $JD_n^{(\nu)}=D_n^{(\nu)}J$, where as above $J$ is the 
matrix with $1$s along the antidiagonal. 
\end{remark}


\section{The expansion of Gegenbauer polynomials\\ in matrix-valued Gegenbauer polynomials}\label{sec:expansionMV-scalarGegenbauer}

The goal of this section is to establish two special cases of 
a connection formula between the matrix-valued Gegenbauer polynomials and 
the scalar Gegenbauer polynomials.
To start with, we define the matrices $F^{(\nu)}_{k,n}$, $k\in \{0,1,\dots, n\}$, by 
\begin{equation}\label{eq:defFnukn}
\hat{P}_{n}^{(\nu)}(x)  = \sum_{k=0}^{n} F^{(\nu)}_{k,n} \, 
C^{(\nu+2\ell)}_{n-k}(x) 
\end{equation}
and in dual setting we define 
$G^{(\nu)}_{k,n}$, $k\in \{0,1,\dots, n\}$, by 
\begin{equation}\label{eq:defGnurm1}
C^{(\nu)}_m(x) \Id = \sum_{r=0}^m G^{(\nu)}_{r,m} \hat{P}_{m-r}^{(\nu)}(x).
\end{equation}
We show that the summation ranges in \eqref{eq:defFnukn}, \eqref{eq:defGnurm1} are bounded by $2\ell$, so that the number of non-zero terms in 
\eqref{eq:defFnukn}, \eqref{eq:defGnurm1} is at most the size of the 
matrix-valued polynomials. 
We first describe the relation arising from 
\eqref{eq:derivativesymmetricPn}.

\begin{lemma}\label{lem:recurrencerelationsGF}
For $n,m\in \N$, $n\geq 1$, $m\geq 1$, we have
\begin{alignat*}{2}
F^{(\nu)}_{k,n} &= \frac{n}{2(\nu+2\ell)}\, F^{(\nu+1)}_{k,n-1},
&\quad& 0\leq k \leq n-1, \\
G^{(\nu)}_{r,m} &= \frac{2\nu}{m-r}\, G^{(\nu+1)}_{r,m-1}, 
&\quad& 0\leq r \leq m-1.
\end{alignat*}
\end{lemma}

\begin{proof}
Differentiate \eqref{eq:defFnukn} using 
$\frac{d}{dx}C^{(\la)}_n(x) = 2\la\, C^{(\la+1)}_{n-1}(x)$
and \eqref{eq:derivativesymmetricPn} to find out that
\begin{align*}
n \,\hat{P}_{n-1}^{(\nu+1)}(x)
&= \sum_{k=0}^{2\ell} F^{(\nu)}_{k,n} \, 
2(\nu+2\ell)\, C^{(\nu+1+2\ell)}_{n-1-k}(x)
\\
&= 2(\nu+2\ell)\sum_{k=0}^{2\ell} 
\bigl( D_n^{(\nu)}\bigr)^{-1}\, D_{n-1}^{(\nu+1)}\, F^{(\nu)}_{k,n} \, 
 C^{(\nu+1+2\ell)}_{n-1-k}(x).
\end{align*}
On the other hand, applying \eqref{eq:defFnukn} with $(n,\nu)$ replaced 
by $(n-1,\nu+1)$ gives 
\begin{align*}
 F^{(\nu+1)}_{k,n-1}  = 
 \frac{2(\nu+2\ell)}{n} \,  F^{(\nu)}_{k,n}
\end{align*}
by uniqueness of the expansion. This proves the first statement, 
and the second one follows analogously.
\end{proof}

In particular, it follows from Lemma \ref{lem:recurrencerelationsGF} that
\begin{equation}\label{eq:reductionGrmtoGmm}
G^{(\nu)}_{m,m+k} = \frac{2^k \, (\nu)_k}{k!}\, G^{(\nu+k)}_{m,m}
\quad\text{and}\quad 
F^{(\nu)}_{m,m+k} = \frac{(m+1)_k}{2^k\, (\nu+2\ell)_k}\, 
F^{(\nu+k)}_{m,m}
\end{equation}
for $m,k\in \N$. Therefore, viewing $(G^{(\nu)}_{r,m})_{r,m\in \N}$ 
and $(F^{(\nu)}_{r,m})_{r,m\in \N}$ as infinite matrices we see that they are upper triangular, with \eqref{eq:reductionGrmtoGmm} 
showing that each element in the $r$-th row is determined by the 
element on the diagonal in the $r$-th row. 

Next we look at $G^{(\nu)}_{m,m}$ and $F^{(\nu)}_{m,m}$.  
From \eqref{eq:reductionGrmtoGmm} we see that the upper bound of the sum in 
\eqref{eq:defGnurm1}, respectively \eqref{eq:defFnukn}, can 
be replaced by $2\ell\wedge m =\min(2\ell,m)$, respectively $2\ell\wedge n = \min(2\ell,n)$, if we show that 
$G^{(\nu)}_{m,m} = F^{(\nu)}_{m,m} = 0$ for $m>2\ell$. 

\begin{lemma}\label{lem:Fnummis0formgreater2ell}
$F^{(\nu)}_{m,m} = 0$ for $m>2\ell$.
\end{lemma}

\begin{proof}
Note that the $(i,j)$-th entry of the matrix $F^{(\nu)}_{m,m}$ is a multiple of the integral 
\begin{equation}\label{eq:integralFnumm}
\int_{-1}^1 \bigl(P^{(\nu)}_m(x)\bigr)_{i,j} \, (1-x^2)^{\nu+2\ell-\frac12}\, dx
\end{equation}
and we need to show that this integral vanishes for $m>2\ell$. By  
\cite[Thm.~3.4]{KoeldlRR} we have 
\[
\bigl(P^{(\nu)}_m(x)\bigr)_{i,j} = 
\sum_{p=j}^{2\ell\wedge m+i} z_p\, C^{(\nu+p)}_{m+i-p}(x) C^{(1-\nu-p)}_{p-j}(x) 
\]
for some explicit constants $z_p$, so that the integral \eqref{eq:integralFnumm} equals
\begin{align*}
\sum_{p=j}^{2\ell\wedge m+i} z_p\, 
\int_{-1}^1 C^{(\nu+p)}_{m+i-p}(x) \bigl\{ C^{(1-\nu-p)}_{p-j}(x) 
(1-x^2)^{2\ell-p} \bigr\} (1-x^2)^{\nu+p-\frac12}\, dx. 
\end{align*}
Using the orthogonality \eqref{eq:Gegenbauer-orthorelations} and 
the fact that the term in the parentheses is a polynomial of 
degree $4\ell-p-j$, we see that the $p$-th term vanishes for 
$m+i-p> 4\ell-p-j$, that is, for $m>4\ell -i-j$. Therefore, the integral \eqref{eq:integralFnumm} vanishes for $m>2\ell$ when $i+j\leq 2\ell$.

Because $J$ commutes with $W^{(\nu)}(x)$ for all $x$, see  \cite[Prop.~2.6]{KoeldlRR}, we also have that the monic polynomials
$P^{(\nu)}_n(x)$ commute with $J$ by \cite[Lemma~3.1(2)]{KoelR}. 
In particular, $(P^{(\nu)}_m(x)\bigr)_{2\ell-i,2\ell-j}=
(P^{(\nu)}_m(x)\bigr)_{i,j}$ and 
the integral in \eqref{eq:integralFnumm} is zero for $m>2\ell$ 
for all $(i,j)$. 
\end{proof}

In principle the approach of Lemma \ref{lem:Fnummis0formgreater2ell} can also be used to calculate 
$F^{(\nu)}_{m,m}$ for $m\leq 2\ell$; this gives an explicit but rather complicated expression for $F^{(\nu)}_{k,m}$ with the help of \eqref{eq:reductionGrmtoGmm}.
In Theorem \ref{thm:expansionFnumr} we give a simple formula for $F^{(\nu)}_{k,m}$. 
We leave it open whether the summation formula induced from these two different calculations is of interest on its own.

For the matrices $G^{(\nu)}_{k,m}$ we first recall that we can rewrite
\eqref{eq:defmonicMVorthopol} as 
\begin{equation}\label{eq:orthorelhatPnKRR17}
\int_{-1}^1 \hat{P}_n^{(\nu)}(x) \, W(x)\, (\hat{P}_m^{(\nu)}(x))^{\ast}\, dx = \de_{m,n} \hat{H}_n^{(\nu)}
\end{equation}
with $\hat{H}_n^{(\nu)} = D_n^{(\nu)}\, H_n^{(\nu)}\, (D_n^{(\nu)})^\ast$ being a diagonal matrix as well. 
Then
\begin{align*}
\int_{-1}^1 C_m^{(\nu)}(x) \, W(x)\, (D^{(\nu)}_0)^\ast \, dx
&= \int_{-1}^1 \sum_{r=0}^m G^{(\nu)}_{r,m} \hat{P}_{m-r}^{(\nu)}(x) \, W(x)\, (\hat{P}^{(\nu)}_0(x))^\ast \, dx
\\
&= G^{(\nu)}_{m,m}\, \hat{H}_0^{(\nu)}
\end{align*}
and
\begin{equation}
\qquad G^{(\nu)}_{m,m}\, D_0^{(\nu)}  H_0^{(\nu)} = 
\int_{-1}^1 C_m^{(\nu)}(x) \, W(x)\,  dx.
\label{eq:calcGnumm-1}
\end{equation}
Note that $D_0^{(\nu)}  H_0^{(\nu)}$ is an explicit invertible diagonal matrix, so that all we need to do is to calculate the matrix entries of
\[
\int_{-1}^1 C_m^{(\nu)}(x) \, W(x)_{i,j}\,  dx
\]
to determine $G^{(\nu)}_{m,m}$.
Observe that $W(x)_{i,j}= W(x)_{j,i}$, since $W$ is Hermitian and real-valued for $x\in (-1,1)$.
Furthermore, $W(x)_{i,j}= W(x)_{2\ell-i,2\ell-j}$ by \cite[Prop.~2.6]{KoeldlRR}. 
This reduces our calculation to evaluating the integral
\[
\int_{-1}^1 C_m^{(\nu)}(x) \, W(x)_{i,j}\, dx
\]
for the case $i\geq j$, $i+j\leq 2\ell$.

\begin{lemma}\label{lem:integralGegenbauerWij} 
Assume $j\geq i$, $i+j\leq 2\ell$. Then 
\[
\int_{-1}^1 C_m^{(\nu)}(x) \, W(x)_{i,j}\, dx = 0
\]
when $m<j-i$ or $i+j\not\equiv m \bmod 2$ or $m> 2\ell$. 
\end{lemma}

In particular, Lemma \ref{lem:integralGegenbauerWij} shows that the 
$r$-th term of the sum \eqref{eq:defGnurm1} vanishes when $r> 2\ell\wedge m$. 

\begin{proof}
Recall formula \eqref{eq:defweight}.
Using the orthogonality  relations \eqref{eq:Gegenbauer-orthorelations} and the fact that the matrix-valued Gegenbauer polynomials are symmetric, the statement follows. 
\end{proof}

\begin{thm}\label{thm:expansionGnumr}
Define
\begin{align*}
& \phi(\nu;i,j,r)
=\phi_\ell(\nu;i,j,r)
\\ &\;
=\binom{2\ell}r\binom r{\frac12(r+i-j)}\binom{2\ell-r}{\frac12(i+j-r)}{\binom{2\ell}i}^{-1}{\binom{2\ell}j}^{-1}
(\nu-r+j)(\nu+2\ell-r-j)
\\ &\;\quad\times
\,\frac{\Gamma(\nu+2\ell-r)\,\Gamma(\nu-\frac12(r-i+j))\,\Gamma(\nu-\frac12(r+i-j))}
{\Gamma(\nu+1-\frac12(r-i-j))\,\Gamma(\nu+2\ell+1-\frac12(r+i+j))}
\end{align*}
if $i+j\equiv r\bmod2$ and $\phi(\nu;i,j,r)=0$ otherwise.
Then for $0\leq r \leq 2\ell$, $m\in \N$, the matrices  
\begin{equation*}
(G_{r,m}^{(\nu)})_{i,j}=\frac{2^{m-r}}{(m-r)!\,\Gamma(\nu)}\cdot\phi(\nu+m;i,j,r)
\qquad i,j\in\{0,1,\dots,2\ell\}, 
\end{equation*}
satisfy 
\begin{equation*}
C^{(\nu)}_m(x) \Id = \sum_{r=0}^{m \wedge 2\ell} G^{(\nu)}_{r,m} \hat{P}_{m-r}^{(\nu)}(x).
\end{equation*}
\end{thm}

\begin{proof}
The zero entries of the matrix in \eqref{eq:calcGnumm-1} given in Lemma \ref{lem:integralGegenbauerWij} can be ignored. 
Employing the formula
\begin{equation*}
(H^{(\nu)}_0)_{j,j} 
= \sqrt{\pi}\, \frac{\Ga(\nu+\frac12)}{\Ga(\nu+1)}\, 
(2\ell+\nu) \frac{j!\, (2\ell-j)!\, (\nu+1)_{2\ell}}{(2\ell)!\, 
(\nu+1)_j\, (\nu+1)_{2l-j}}
\end{equation*}
for the squared norm at $n=0$ and \eqref{eq:defweight} for $\al_k(i,j)$, using 
the expression for $D_n^{(\nu)}$ from \eqref{eq:def-symmetric-Pn} 
and the orthogonality relations \eqref{eq:Gegenbauer-orthorelations} 
in \eqref{eq:calcGnumm-1}, we can evaluate $G^{(\nu)}_{m,m}$. 
Then the formula for $G^{(\nu)}_{m,r}$ follows from the shift property \eqref{eq:reductionGrmtoGmm}.
A straightforward calculation gives the result. 
\end{proof}

As a corollary to the proof of Theorem \ref{thm:expansionGnumr}, we find symmetry properties for $G^{(\nu)}_{r,m}$. 

\begin{cor}
We have $G^{(\nu)}_{r,m}J= JG^{(\nu)}_{r,m}$, and the matrix
$G^{(\nu)}_{r,m} D^{(\nu)}_0H^{(\nu)}_0$ is symmetric.
\end{cor}

\begin{proof}
The explicit expression of $D^{(\nu)}_n$ in \eqref{eq:def-symmetric-Pn} shows that $D^{(\nu)}_n$ commutes with $J$. Furthermore, $W(x)$ commutes with $J$ by 
\cite[Prop.~2.6]{KoeldlRR}. This implies that $J$ commutes with 
$P^{(\nu)}_n(x)$ and $H_n^{(\nu)}$ by \cite[Lemma~3.1]{KoelR} and the fact that $J^\ast = J$. In turn, this means that $J$ also commutes with 
$\hat{P}^{(\nu)}_n(x)$ and $\hat{H}_n^{(\nu)}$. With the help of \eqref{eq:calcGnumm-1} we deduce that
\begin{align*}
G^{(\nu)}_{m,m} = 
\int_{-1}^1 C_m^{(\nu)}(x) \, W^{(\nu)}(x)\, dx\, (D^{(\nu)}_0)^\ast \bigl(\hat{H}_0^{(\nu)}\bigr)^{-1} 
 = 
\int_{-1}^1 C_m^{(\nu)}(x) \, W^{(\nu)}(x)\, dx\, 
(H^{(\nu)}_0)^{-1} (D^{(\nu)}_0)^{-1},
\end{align*}
so that $G^{(\nu)}_{m,m}$ commutes with $J$. By \eqref{eq:reductionGrmtoGmm} we see that $G^{(\nu)}_{r,m}$ commutes with $J$. 
This also implies that 
$G^{(\nu)}_{m,m} D^{(\nu)}_0H^{(\nu)}_0$ is a symmetric matrix, hence 
the statement for $G^{(\nu)}_{r,m}$ follows from~\eqref{eq:reductionGrmtoGmm}.  
\end{proof}

Our next objective is to determine $F^{(\nu)}_{k,n}$ using Lemma 
\ref{lem:Fnummis0formgreater2ell}. We follow an indirect approach
and prove \eqref{eq:defFnukn} by showing that the right-hand
side of \eqref{eq:defFnukn} satisfies the three-term recurrence relation for the polynomials $\hat{P}^{(\nu)}_n(x)$. 
In other words, we want to show that these polynomials satisfy the three-term recurrence relation
\begin{equation}\label{eq:3TTR-symmPn}
x\hat{P}^{(\nu)}_{n}(x)= \hat{A}_n^{(\nu)} \hat{P}^{(\nu)}_{n+1}(x)+\hat{B}_n^{(\nu)}\hat{P}^{(\nu)}_{n}(x)+
\hat{C}^{(\nu)}_{n}\hat{P}^{(\nu)}_{n-1}(x) 
\end{equation}
with $\hat{A}_n^{(\nu)}= D^{(\nu)}_{n} (D^{(\nu)}_{n+1})^{-1}$, $\hat{B}_n^{(\nu)} = D^{(\nu)}_{n} B_n^{(\nu)} (D^{(\nu)}_{n})^{-1}$ and $\hat{C}_n^{(\nu)} = D^{(\nu)}_{n} C_n^{(\nu)} (D^{(\nu)}_{n-1})^{-1}$ with the matrices $B_n^{(\nu)}$, 
$C_n^{(\nu)}$ as in \eqref{eq:3TTR-Pn}. Note that 
$\hat{A}_n^{(\nu)}$ and $\hat{C}_n^{(\nu)}$ are symmetric, however 
$\hat{B}_n^{(\nu)}$ is not in general.

\begin{thm}\label{thm:expansionFnumr}
For $i,j,k\in\{0,1,\dots,2\ell\}$, define
\begin{align*}
\ga(\nu;i,j,k)
&=\ga_\ell(\nu;i,j,k)
\\
&=(-1)^k(\nu+2\ell)(\nu+2\ell-k)\binom{2\ell}k\binom k{\frac12(k+i-j)}\binom{2\ell-k}{\frac12(i+j-k)}
\\ &\,\quad\times
\frac{\Ga(\nu-\frac12(k-i-j))\,\Ga(\nu+2\ell-\frac12(k+i+j))}
{\Ga(\nu)\,\Ga(\nu+2\ell+1-\frac12(k-i+j))\,\Ga(\nu+2\ell+1-\frac12(k+i-j))}
\end{align*}
if $i+j\equiv k\bmod2$ and $\gamma(\nu;i,j,k)=0$ otherwise. 
Let the matrix $F^{(\nu)}_{k,n}$ be given by  
\begin{align*}
(F^{(\nu)}_{k,n})_{i,j}
=\frac{n!\,\Ga(\nu+2\ell)}{2^n} \ga(\nu+n;i,j,k);
\end{align*}
then 
\begin{equation*}
\hat{P}_{n}^{(\nu)}(x)  = \sum_{k=0}^{n\wedge 2\ell} F^{(\nu)}_{k,n} \, 
C^{(\nu+2\ell)}_{n-k}(x). 
\end{equation*}
\end{thm}

Since $\binom{n}{k}$ is non-zero only for $n,k\in \N$ with  
$0\leq k\leq n$, we see that $F^{(\nu)}_{0,n}$ is a diagonal matrix.
Similarly, $F^{(\nu)}_{1,n}$ is zero on the diagonal and  has non-zero sub- and superdiagonals.

Since the matrices $F^{(\nu)}_{k,n}$ are symmetric, the following corollary is immediate. 

\begin{cor}\label{cor:thm:expansionFnumr}
The polynomials $\hat{P}_{n}^{(\nu)}(x)$ are symmetric, i.e. 
$\bigl(\hat{P}_{n}^{(\nu)}(x)\bigr)^t=\hat{P}_{n}^{(\nu)}(x)$.
In particular, $\bigl(P_{n}^{(\nu)}(x)\bigr)^t=
D_n^{(\nu)} \, P_{n}^{(\nu)}(x) (D_n^{(\nu)})^{-1}$. 
\end{cor}

\begin{proof}[Proof of Theorem \ref{thm:expansionFnumr}]
Lemma \ref{lem:Fnummis0formgreater2ell} gives us the 
bound $n\wedge 2\ell$ for the number of terms in the expansion. We prove the latter 
by showing that the right-hand side satisfies the same 
three-term recurrence \eqref{eq:3TTR-symmPn} and initial conditions for $n=-1$ and $n=0$. For $n=-1$, we have that the 
right-hand side is the zero matrix. For $n=0$, the right-hand side reduces to the diagonal matrix $F^{(\nu)}_{0,0}$ and by inspection 
\begin{equation*}
\bigl( F^{(\nu)}_{0,0}\bigr)_{i,i} = \Ga(\nu+2\ell)\, 
(\nu+2\ell)^2 \,\binom{2\ell}{i}\,  
\frac{\Ga(\nu+i)\, \Ga(\nu+2\ell-i)}{\Ga(\nu)\,  \Ga(\nu+2\ell+1)^2} = \binom{2\ell}{i}\, \frac{(\nu)_i}{(\nu+2\ell-i)_i} = (D^{(\nu)}_0)_{i,i}.
\end{equation*}
Thus, the initial values match, and it remains to show that the three-term recurrence relation \eqref{eq:3TTR-symmPn} is
satisfied by $\sum_{k=0}^{n\wedge 2\ell} F^{(\nu)}_{k,n} \, 
C^{(\nu+2\ell)}_{n-k}(x)$. 

Using the three-term recurrence relation \eqref{eq:Gegenbauer-3TRR} for the Gegenbauer polynomials 
and the explicit expressions we see that $(i,j)$-entry of 
$x\, \sum_{k=0}^{n\wedge 2\ell} F^{(\nu)}_{k,n} \, 
C^{(\nu+2\ell)}_{n-k}(x)$ consists of two sums in terms of the 
Gegenbauer polynomials $C^{(\nu+2\ell)}_{n-k}(x)$ for $k\geq 0$. Similarly, 
since the matrices $\hat{A}^{(\nu)}_{n}$, $\hat{B}^{(\nu)}_{n}$  and $\hat{C}^{(\nu)}_{n}$ that appears on the right-hand side of \eqref{eq:3TTR-symmPn} are either diagonal or have two non-zero diagonals, we see that the right-hand side involves four sums in terms of Gegenbauer polynomials $C^{(\nu+2\ell)}_{n-k}(x)$ with $k\geq 0$. 
Comparing the coefficients of $C^{(\nu+2\ell)}_{n-k}(x)$ on both sides, we see that we need the equality
\begin{align*}
&
\frac{n-k}{\nu+2\ell+n-k-1}\, \ga(\nu+n;i,j,k+1)
+\frac{2(\nu+2\ell)+n-k}{\nu+2\ell+n-k+1}\, \ga(\nu+n;i,j,k-1)
\\ &\quad
= \frac{(n+1)(\nu+n)(\nu+2\ell+n)}{(\nu+n+i)(\nu+2\ell+n-i)}\, \ga(\nu+n+1;i,j,k+1)
\\ &\quad\quad 
+ \frac{(\nu+i-1)(2\ell-i+1)}{(\nu+n+i)(\nu+2\ell+n-i)}\, \ga(\nu+n;i-1,j,k)
\\ &\quad\quad 
+\frac{(i+1)(\nu+2\ell-i-1)}{(\nu+n+i)(\nu+2\ell+n-i)}\, \ga(\nu+n;i+1,j,k)
\\ &\quad\quad 
+ \frac{(\nu+2\ell+n)(2\nu+2\ell+n-1)}{(\nu+n+i)(\nu+2\ell+n-i)(\nu+2\ell+n-1)}\, \ga(\nu+n-1;i,j,k-1)
\end{align*}
to be valid, where we have divided by the normalising constant 
$n!\,\Ga(\nu+2\ell)/2^{n+1}$. This identity is trivially true in case $i+j\not\equiv k\bmod2$, since all six terms equal zero. In the general case, we divide by $\ga(\nu+n;i,j,k+1)$ and the 
required identity becomes an identity involving rational functions in the parameters. This rational identity is checked (by computer algebra) to be valid. 
\end{proof}

The matrices $G^{(\nu)}_{r,m}$ and $F^{(\nu)}_{k,n}$ determined
in Theorems~\ref{thm:expansionGnumr} and \ref{thm:expansionFnumr} are related to expansion and summation formulae, which we state below. 

\begin{cor}\label{cor:thmsonGnurmandFnukn}
The following expansion holds for the matrix-valued polynomials
$\hat{P}^{(\nu)}_n(x)$\textup:
\begin{equation*}
\hat{P}^{(\nu)}_n(x) = \sum_{t=0}^{n\wedge 4\ell} 
M^{(\nu)}_t \hat{P}^{(\nu+2\ell)}_{n-t}(x), 
\quad\text{where}\;\;
M^{(\nu)}_t = \sum_{k=0\vee t-2\ell}^{n\wedge 2\ell} F^{(\nu)}_{k,n} G^{(\nu+2\ell)}_{t-k,n-k}.
\end{equation*}
\end{cor}

Corollary \ref{cor:thmsonGnurmandFnukn} follows immediately 
by first applying Theorem \ref{thm:expansionFnumr} and next 
Theorem~\ref{thm:expansionGnumr}. Note that 
Corollary \ref{cor:thmsonGnurmandFnukn} is a matrix analogue of a very specific case of~\eqref{eq:connectionGegenbauerpol-integer}. 

\begin{cor}\label{cor:thmsonGnurmandFnukn2}
The following double summation result holds\textup: for $s\in \N$ with 
$0\leq s\leq \lfloor m/2\rfloor$ and indices $i,j$, we have
\begin{align*}
\sum_{p=0}^{2\ell}
\sum_{r=0\vee 2s-2\ell}^{m\wedge 2\ell}
(G^{(\nu)}_{r,m})_{i,p}  (F^{(\nu)}_{2s-r,m-r})_{p,j} 
= \de_{i,j}
\,\frac{\nu+2\ell+m-2s}{\nu+2\ell}
\,\frac{(\nu)_{m-s}}{(\nu+2\ell+1)_{m-s}} 
\,\frac{(-2\ell)_s}{s!}.
\end{align*}
\end{cor}

\begin{proof}
First apply Theorem \ref{thm:expansionGnumr} and then Theorem \ref{thm:expansionFnumr} to deduce 
\begin{equation*}
C^{(\nu)}_m(x)\Id = 
\sum_{r=0}^{m\wedge 2\ell} \sum_{k=0}^{ (m-r)\wedge 2\ell}
G^{(\nu)}_{r,m} F^{(\nu)}_{k,m-r} \, C^{(\nu+2\ell)}_{m-p};
\end{equation*}
this connection formula of the scalar Gegenbauer polynomials is known: see \eqref{eq:connectionGegenbauerpol-integer} with $N=2\ell$. 
Comparing the two we conclude that
\[
\sum_{r=0\vee t-2\ell}^{m\wedge 2\ell} G^{(\nu)}_{r,m} F^{(\nu)}_{t-r,m-r}=\bold0
\]
for $t$ odd (which is already clear from the definitions in Theorems~\ref{thm:expansionGnumr} and~\ref{thm:expansionFnumr}), and
\begin{equation*}
\sum_{r=0\vee 2s-2\ell}^{m\wedge 2\ell} G^{(\nu)}_{r,m} F^{(\nu)}_{2s-r,m-r}=
\frac{\nu+2\ell+m-2s}{\nu+2\ell}
\,\frac{(\nu)_{m-s}}{(\nu+2\ell+1)_{m-s}} 
\,\frac{(-2\ell)_s}{s!} \, \Id
\end{equation*}
for $t=2s$ even.
Taking the $(i,j)$-th entry we obtain the desired result. 
\end{proof}

The identity in Corollary \ref{cor:thmsonGnurmandFnukn2} gives an example of a matrix hypergeometric summation formula;
we are not aware of its proof using classical hypergeometric identities.


\section{Differential and difference equations\\ related to matrix-valued Gegenbauer polynomials}
\label{sec:differenceeqtFknnu}

Since the matrix-valued Gegenbauer polynomials $\hat{P}^{(\nu)}_n$ are 
symmetric by Corollary \ref{cor:thm:expansionFnumr}, we can 
take the transposed version of identities for 
the polynomials $\hat{P}^{(\nu)}_n$ and compare the resulting identity to the original one. This gives various identities for the matrix-valued polynomials $\hat{P}^{(\nu)}_n$;
with the help of Theorem~\ref{thm:expansionFnumr} we can rewrite those in terms of recurrence relations for the 
matrices $F^{(\nu)}_{k,n}$. Therefore, the fact that 
$\hat{P}^{(\nu)}_n$ are symmetric polynomials gives 
a series of identities which can be viewed as mixed differential-difference equations, where the differential and difference operators can act from both sides. This procedure can be performed for any set of 
difference, respectively differential, equations that occur in the left, respectively right, Fourier algebras associated to the matrix-valued Gegenbauer polynomials, see 
\cite{CaspY} for the definition. The results obtained do \emph{not} have any classical analogues, since in the scalar case all commutators vanish. The resulting identities have a true matrix nature. 

In this section we outline this procedure for three explicit situations. First we consider the 
three-term recursion for the matrix-valued Gegenbauer polynomials $\hat{P}^{(\nu)}_n$, and next we deal with the two matrix differential operators that have the matrix-valued Gegenbauer polynomials $\hat{P}^{(\nu)}_n$ as eigenfunctions, see \cite[Thms.~2.3, 3.2]{KoeldlRR}.


We start with the three-term recursion \eqref{eq:3TTR-symmPn}, which is the basic example of an element in the left Fourier algebra. The symmetry of the matrix-valued orthogonal polynomials $\hat{P}_{n}^{(\nu)}(x)$ shows that the polynomials also satisfy a three-term recurrence relation with matrix multiplication by matrices depending on $n$ from the right, i.e. 
\begin{equation} \label{eq:3TTR-symmPnT}
x\hat{P}^{(\nu)}_{n}(x)
=  \hat{P}^{(\nu)}_{n+1}(x) \hat{A}^{(\nu)}_{n} + 
\hat{P}^{(\nu)}_{n}(x) (\hat{B}_n^{(\nu)})^t +
\hat{P}^{(\nu)}_{n-1}(x)\hat{C}^{(\nu)}_{n} 
\end{equation}
taking into account that the diagonal matrices $\hat{A}^{(\nu)}_{n}$ and $\hat{C}^{(\nu)}_{n}$ are automatically symmetric.

\begin{prop}\label{prop:recursionfrom3TRsymmetricPn}
The symmetric matrix-valued Gegenbauer polynomials satisfy 
\begin{equation*} 
[\hat{P}^{(\nu)}_{n+1}(x), \hat{A}^{(\nu)}_{n}]
+ \hat{P}^{(\nu)}_{n}(x) \bigl( \hat{B}_n^{(\nu)}\bigr)^t
- \hat{B}_n^{(\nu)} \hat{P}^{(\nu)}_{n}(x) 
+ [\hat{P}^{(\nu)}_{n-1}(x), \hat{C}^{(\nu)}_{n}]=0.
\end{equation*}
In turn, the matrices $F^{(\nu)}_{k,n}$ defined in Theorem \ref{thm:expansionFnumr} satisfy 
\begin{align*}
[F^{(\nu)}_{k+1,n+1}, \hat{A}^{(\nu)}_n] 
+ F^{(\nu)}_{k,n} (\hat{B}^{(\nu)}_n)^t 
- \hat{B}^{(\nu)}_n F^{(\nu)}_{k,n} 
+ [F^{(\nu)}_{k-1,n-1}, \hat{C}^{(\nu)}_n]
= 0,
\end{align*}
with the convention that $F^{(\nu)}_{k,n}=0$ for $k>n\wedge 2\ell$ or for $k<0$, and $\hat{C}^{(\nu)}_0=0$.  
\end{prop}

Entry-wise the recursion for $F^{(\nu)}_{k,n}$ boils down to 
\begin{multline*}
\bigl((\hat{A}^{(\nu)}_n)_{j,j} - (\hat{A}^{(\nu)}_n)_{i,i}\bigr) (F^{(\nu)}_{k+1,n+1})_{i,j}
+ (\hat{B}^{(\nu)}_n)_{j,j-1} (F^{(\nu)}_{k,n})_{i,j-1} +
(\hat{B}^{(\nu)}_n)_{j,j+1} (F^{(\nu)}_{k,n})_{i,j+1}
\\
- (\hat{B}^{(\nu)}_n)_{i,i-1} (F^{(\nu)}_{k,n})_{i-1,j} 
- (\hat{B}^{(\nu)}_n)_{i,i+1} (F^{(\nu)}_{k,n})_{i+1,j} 
+ \bigl((\hat{C}^{(\nu)}_n)_{j,j} - (\hat{C}^{(\nu)}_n)_{i,i}\bigr) (F^{(\nu)}_{k-1,n-1})_{i,j} = 0
\end{multline*}
with the explicit matrix entries for $\hat{A}^{(\nu)}_n$,
$\hat{B}^{(\nu)}_n$ and $\hat{C}^{(\nu)}_n$ 
recorded in \eqref{eq:3TTR-Pn}, \eqref{eq:3TTR-symmPn}. 

\begin{proof}
The first part follows by subtracting \eqref{eq:3TTR-symmPn} from \eqref{eq:3TTR-symmPnT}.
Note that this is a polynomial identity of degree $n$, since the leading coefficient $D^{(\nu)}_{n+1}$ of $\hat{P}^{(\nu)}_{n+1}$ commutes with $\hat{A}^{(\nu)}_{n}$ as both are diagonal. 

The statement for the matrices $F^{(\nu)}_{k,n}$ then follows by plugging Theorem \ref{thm:expansionFnumr} in the identity from the first part.
This procedure leads to an expansion in terms of Gegenbauer polynomials $C^{(\nu+2\ell)}_m$; collecting the coefficients of a Gegenbauer polynomial of fixed degree gives the result. 
\end{proof}

Our next example arises from the second-order matrix differential operator of hypergeometric type for which the matrix-valued Gegenbauer polynomials are eigenfunctions.
At the same time, the matrix-valued Gegenbauer polynomials are eigenfunctions for a first-order matrix differential equation.
These two matrix differential operators are in the right Fourier algebra for the matrix-valued Gegenbaure polynomials. We recall the operators explicitly from \cite[Thm.~3.2]{KoeldlRR}. 

The second-order matrix hypergeometric differential operator for which the polynomials are eigenfunctions is given as follows, see
\cite[Thms.~2.3, 3.2]{KoeldlRR}, where we switched to the notation $\cD^{(\nu)}$ in order to avoid confusion with the diagonal matrix $D^{(\nu)}_n$ used in this paper:
\begin{gather}\label{eq:DOD-Pn}
P^{(\nu)}_n \cdot \cD^{(\nu)} = \Lambda_n(\cD^{(\nu)})\, P^{(\nu)}_n,
\quad
\Lambda_n(\cD^{(\nu)}) =  -n(2\ell+2\nu+n)\Id-V,
\\ \intertext{with}
\cD^{(\nu)} = \frac{d^2}{dx^2}\, (1-x^2)\Id + \frac{d}{dx} \, (C -x(2\ell+2\nu+1)\Id) - V,
\nonumber\\
C =\sum_{i=0}^{2\ell-1} (2\ell-i) E_{i,i+1} + \sum_{i=1}^{2\ell} i E_{i,i-1},
\quad 
V= -\sum_{i=0}^{2\ell} i(2\ell-i) E_{i,i}.
\nonumber
\end{gather}
It follows that
\begin{equation}\label{eq:DOD-symPn}
\hat{P}^{(\nu)}_n \cdot \cD^{(\nu)} = D_n^{(\nu)} 
\Lambda_n(\cD^{(\nu)})(D_n^{(\nu)})^{-1}\, \hat{P}^{(\nu)}_n
= \Lambda_n(\cD^{(\nu)})\, \hat{P}^{(\nu)}_n,
\end{equation}
since $D_n^{(\nu)}$ and $\Lambda_n(\cD^{(\nu)})$ commute, being diagonal matrices. 

\begin{prop}\label{prop:recursionfromDODsymmetricPn}
In the notation above, the symmetric matrix-valued Gegenbauer polynomials satisfy 
\begin{equation*} 
\frac{d\hat{P}^{(\nu)}_n}{dx}(x)\, C -C^t \, \frac{d\hat{P}^{(\nu)}_n}{dx}(x) = -2 [V, \hat{P}^{(\nu)}_n(x)].
\end{equation*}
In turn, the matrices $F^{(\nu)}_{k,n}$ defined in Theorem~\ref{thm:expansionFnumr} satisfy 
\[
 F^{(\nu)}_{k,n}C- C^t   F^{(\nu)}_{k,n}
 = \frac{1}{\nu+2\ell+n-k-1}\, [F^{(\nu)}_{k+1,n},V] 
 - \frac{1}{\nu+2\ell+n-k+1}\, [F^{(\nu)}_{k-1,n},V].  
\]
\end{prop}

Observe that the first identity of Proposition \ref{prop:recursionfromDODsymmetricPn} is a matrix-valued polynomial identity of degree $n-1$, since the leading coefficient of $\hat{P}^{(\nu)}_n$ is diagonal and commutes with $V$.
Furthermore, notice that the participating matrices $C$ and $V$ depend on neither degree~$n$ nor parameter~$\nu$.

Entry-wise the recursion for $F^{(\nu)}_{k,n}$ in Proposition \ref{prop:recursionfromDODsymmetricPn} reads
\begin{multline*}
\frac{V_{j,j} - V_{i,i}}{\nu+2\ell+n-k-1}\, (F^{(\nu)}_{k+1,n})_{i,j}
- 
\frac{V_{j,j} - V_{i,i}}{\nu+2\ell+n-k+1}\, (F^{(\nu)}_{k-1,n})_{i,j}
\\
= C_{j-1,j} (F^{(\nu)}_{k,n})_{i,j-1} + C_{j+1,j} (F^{(\nu)}_{k,n})_{i,j+1}
-  C_{i-1,i} (F^{(\nu)}_{k,n})_{i-1,j} - C_{i+1,i} (F^{(\nu)}_{k,n})_{i+1,j}
\end{multline*}
with the explicit matrix entries for $C$ and 
$V$ given above. Again, $C$ and $V$ are independent of $n$ and~$\nu$. 

\begin{proof} 
Equation \eqref{eq:DOD-symPn} translates into
\[
\frac{d^2\hat{P}^{(\nu)}_n}{dx^2}(x)\, (1-x^2)\Id + 
\frac{d\hat{P}^{(\nu)}_n}{dx}(x) \, (C -x(2\ell+2\nu+1)\Id) - \hat{P}^{(\nu)}_n(x) V =
\Lambda_n(\cD^{(\nu)})\, \hat{P}^{(\nu)}_n(x).
\]
We take the transpose of this identity using the fact that $V$ and  
$\Lambda(\cD^{(\nu)})$ are diagonal, hence symmetric, and then subtract one from the other,
keeping in mind that multiples of the identity $\Id$ commute with any matrix:
\[
\frac{d\hat{P}^{(\nu)}_n}{dx}(x)\, C -C^t \, \frac{d\hat{P}^{(\nu)}_n}{dx}(x) - [\hat{P}^{(\nu)}_n(x),V] = 
[\Lambda_n(\cD^{(\nu)}), \hat{P}^{(\nu)}_n(x)].
\]
The commutator on the right-hand side does not depend on~$n$ in the eigenvalue $\Lambda_n(\cD^{(\nu)})$, as all the dependence on~$n$ in $\Lambda(\cD^{(\nu)})$ is in the multiple $-n(2\ell+2\nu+n)\Id$ of the identity; it reduces to $-[V, \hat{P}^{(\nu)}_n(x)]$ and gives the first identity of the proposition. 

For the second identity we implement Theorem \ref{thm:expansionFnumr} in the identity just proven and use $\frac{d}{dx} C^{(\nu+2\ell)}_{n-k}(x) = 2(\nu+2\ell)\, C^{(\nu+2\ell+1)}_{n-k-1}(x)$
(see, e.g., \cite{AndrAR,Aske,Isma,KoekLS}) to obtain 
\begin{align*}
\sum_{k=0}^{n\wedge 2\ell} \bigl( F^{(\nu)}_{k,n}C- C^t   F^{(\nu)}_{k,n}\bigr)\,  2(\nu+2\ell)\, C^{(\nu+2\ell+1)}_{n-k-1}(x)
= -2\sum_{k=0}^{n\wedge 2\ell}  [V, F^{(\nu)}_{k,n}]  \, C^{(\nu+2\ell)}_{n-k}(x).
\end{align*}
The case $N=1$ of \eqref{eq:connectionGegenbauerpol-integer} leads to
\begin{equation}\label{eq:scalarGegenbauerconnected}
\frac{\nu+2\ell+n-k}{\nu+2\ell}\,C^{(\nu+2\ell)}_{n-k}(x)
= C^{(\nu+2\ell+1)}_{n-k}(x) - C^{(\nu+2\ell+1)}_{n-k-2}(x)
\end{equation}
and to an expansion in Gegenbauer polynomials $C^{(\nu+2\ell+1)}_{n-k}$ on the right-hand side.
It remains to compare the coefficients in Gegenbauer polynomials $C^{(\nu+2\ell+1)}_{n-k}$ on both sides to deduce the second identity in the proposition. 
\end{proof}

We perform the same procedure for the first-order matrix differential operator $\cE^{(\nu)}$ contained in the right Fourier algebra. The operator is given by 
\begin{equation}\label{eq:DOE-Pn}
P^{(\nu)}_n \cdot \cE^{(\nu)} = \Lambda_n(\cE^{(\nu)})\, P^{(\nu)}_n,
\quad
\Lambda_n(\cE^{(\nu)})  = A_0^{(\nu)} + n B_1,
\end{equation}
where
\begin{gather*}
\cE^{(\nu)} =  \frac{d}{dx} \, (xB_1+B_0) + A_0^{(\nu)},
\\
2\ell\, B_0 =\sum_{i=0}^{2\ell-1} (2\ell-i) \, E_{i,i+1} - \sum_{i=1}^{2\ell} i\,E_{i,i-1},
\quad 
\ell \, B_1=-\sum_{i=0}^{2\ell} (\ell-i) \, E_{i,i},
\\ 
\ell \, A^{(\nu)}_0=\sum_{i=0}^{2\ell} \big((\ell+1)(i-2\ell)  - (\nu-1)(\ell-i)\big) \, E_{i,i}.
\end{gather*}
Then the symmetric Gegenbauer polynomials satisfy 
\begin{equation}\label{eq:DOE-symPn}
\hat{P}^{(\nu)}_n \cdot \cE^{(\nu)}
= D_n^{(\nu)} \Lambda_n(\cE^{(\nu)})(D_n^{(\nu)})^{-1}\, \hat{P}^{(\nu)}_n
=\Lambda_n(\cE^{(\nu)})\, \hat{P}^{(\nu)}_n,
\end{equation}
since $D_n^{(\nu)}$ and $\Lambda_n(\cE^{(\nu)})$ commute, being diagonal matrices. 

\begin{prop}\label{prop:recursionfromDOEsymmetricPn}
The symmetric matrix-valued Gegenbauer polynomials satisfy 
\begin{equation*} 
x \bigg[\frac{d\hat{P}^{(\nu)}_n}{dx}(x), B_1\bigg] 
+ \frac{d\hat{P}^{(\nu)}_n}{dx}(x)B_0-B_0^t \frac{d\hat{P}^{(\nu)}_n}{dx}(x)
= [2 A_0^{(\nu)} + n B_1, \hat{P}^{(\nu)}_n(x)].
\end{equation*}
Furthermore, the matrices $F^{(\nu)}_{k,n}$ defined in Theorem~\ref{thm:expansionFnumr} satisfy 
\begin{multline*}
\bigg[\frac{F^{(\nu)}_{k-2,n}}{\nu+2\ell+n-k+2} ,(2-k+2\nu+4\ell)B_1-2A^{(\nu)}_0\bigg] 
\\
+\bigg[\frac{F^{(\nu)}_{k,n}}{\nu+2\ell+n-k} ,(2n-k)B_1+2A^{(\nu)}_0\bigg] 
=  2\bigl( B_0^t  F^{(\nu)}_{k-1,n} - F^{(\nu)}_{k-1,n}B_0).
\end{multline*}
\end{prop}

Note that the first identity of Proposition \ref{prop:recursionfromDOEsymmetricPn} is a matrix-valued polynomial identity of degree $n-1$, since the leading coefficient of $\hat{P}^{(\nu)}_n$ is diagonal and commutes with $B_1$ and 
$A_0^{(\nu)}$. We refrain from writing the second identity in terms of matrix coefficients. 

\begin{proof}
Note that \eqref{eq:DOE-symPn} means
\[ 
\frac{d\hat{P}^{(\nu)}_n}{dx}(x) \, (xB_1+B_0) + \hat{P}^{(\nu)}_n(x) A_0^{(\nu)}
= \Lambda_n(\cE^{(\nu)})\, \hat{P}^{(\nu)}_n(x);
\]
we take the transpose of this identity taking into account that $B_1$, $A^{(\nu)}_0$ and  $\Lambda(\cE^{(\nu)})$ are diagonal, hence symmetric.
Subtracting one from the other gives
\begin{align*}
x \bigg[\frac{d\hat{P}^{(\nu)}_n}{dx}(x), B_1\bigg] 
+ \frac{d\hat{P}^{(\nu)}_n}{dx}(x)B_0-B_0^t \frac{d\hat{P}^{(\nu)}_n}{dx}(x)
+[\hat{P}^{(\nu)}_n(x), A_0^{(\nu)}]
= [\Lambda_n(\cE^{(\nu)}), \hat{P}^{(\nu)}_n(x)].
\end{align*}
It remains to collect the two commutators to obtain the first identity.
Using it together with Theorem~\ref{thm:expansionFnumr}, plugging in the derivative of 
the scalar Gegenbauer polynomial and applying the connection formula~\eqref{eq:scalarGegenbauerconnected} leads to
\begin{align*}
&
\sum_{k=0}^{n\wedge 2\ell} [F^{(\nu)}_{k,n},B_1]
\,  2\, x\, C^{(\nu+2\ell+1)}_{n-k-1}(x) 
+ \sum_{k=0}^{n\wedge 2\ell} \bigl( F^{(\nu)}_{k,n}B_0- B_0^t   F^{(\nu)}_{k,n}\bigr)\,  2\, C^{(\nu+2\ell+1)}_{n-k-1}(x)
\\ &\quad
=\sum_{k=0}^{n\wedge 2\ell}  \frac{[2 A_0^{(\nu)} + n B_1, F^{(\nu)}_{k,n}]}{\nu+2\ell+n-k}\bigl( C^{(\nu+2\ell+1)}_{n-k}(x) - C^{(\nu+2\ell+1)}_{n-k-2}(x)\bigr).
\end{align*}
Finally, we use the three-term recursion \eqref{eq:Gegenbauer-3TRR} for the scalar Gegenbauer polynomials to write both sides as expansions in terms of $C^{(\nu+2\ell+1)}_{p}(x)$; comparing the coefficients of these expansions gives us the second identity of the proposition.
\end{proof}


\section{Generating function for matrix-valued Gegenbauer polynomials}
\label{sec:generatingfMVGegenbauerpols}

Several generating functions are known for the scalar Gegenbauer polynomials; they depend on a related normalization factor. The `pure' generating function is simply
\begin{equation}
\label{eq:gen-fun-scalar}
\sum_{n=0}^\infty C_n^{(\lambda)}(x)t^n=\frac1{(1-2xt+t^2)^\lambda}.
\end{equation}
In contrast, generating functions are not known for matrix-valued orthogonal polynomials, in particular, for the matrix-valued Gegenbauer polynomials that we discuss. In this short section we explain how the explicit formulas in Theorem~\ref{thm:expansionFnumr} allow one to write down the generating function for suitably normalized matrix-valued Gegenbauer polynomials when $\ell$ is fixed.
For this purpose consider the polynomials
\[
\wt P_n^{(\nu)}(x)
=\sum_{k=0}^{2\ell} \wt F_{k,n}^{(\nu)}\cdot C_{n-k}^{(\nu+2\ell)}(x),
\]
where
\[
(\wt F_{k,n}^{(\nu)})_{i,j} = \Gamma(\nu+n+2)(\nu+n+1) \gamma(\nu+n;i,j,k).
\]
This normalization of the Gegenbauer polynomials does not affect their symmetry and is chosen in such a way that the entries of new matrices $\wt F_{k,n}^{(\nu)}$ are \emph{polynomials} in $\nu+n$, in fact of degree $\lfloor\ell\rfloor$; indeed, the latter integer counts the maximal number of scalar Gegenbauer polynomials that show up in a linear combination of every entry of the matrix $\wt P_n^{(\nu)}(x)$.
Notice that from \eqref{eq:gen-fun-scalar} we have
\begin{align*}
\sum_{n=0}^\infty(\lambda+n)^jC_n^{(\lambda)}(x)t^n
&= t^{-\lambda}\sum_{n=0}^\infty(\lambda+n)^jC_n^{(\lambda)}(x)t^{\lambda+n}
\\
&= t^{-\lambda}\bigg(t\frac{\d}{\d t}\bigg)^j\sum_{n=0}^\infty C_n^{(\lambda)}(x)t^{\lambda+n}
\\
&= t^{-\lambda}\bigg(t\frac{\d}{\d t}\bigg)^j\bigg(\frac t{1-2xt+t^2}\bigg)^\lambda;
\end{align*}
for example,
\[
\sum_{n=0}^\infty(\lambda+n)C_n^{(\lambda)}(x)t^n=\frac{\lambda(1-t^2)}{(1-2xt+t^2)^{\lambda+1}}.
\]
Using these formulas for $j=0,1,\dots,\lfloor\ell\rfloor$ we can write explicitly the generating series
\[
M(x;t)=\sum_{n=0}^\infty\wt P_n^{(\nu)}(x)t^n
\]
for every particular choice of~$\ell$.
We will limit ourselves to the illustration for $\ell=1$, when the Gegenbauer polynomials are $3\times3$ matrices.
In this case we obtain
\begin{align*}
\wt F_{0,n}^{(\nu)}
&= \left.\begin{pmatrix}
\lambda-1 & 0 & 0 
\\
0 & 2\lambda-4  & 0 
\\
0 & 0 & \lambda-1
\end{pmatrix}\right|_{\lambda=(\nu+2)+n},
\displaybreak[2]\\
\wt F_{1,n}^{(\nu)}
&= \left.\begin{pmatrix}
 0 & -2\lambda & 0 
\\
-2\lambda & 0 & -2\lambda 
\\
0 & -2\lambda & 0 
\end{pmatrix}\right|_{\lambda=(\nu+2)+(n-1)},
\displaybreak[2]\\
\wt F_{2,n}^{(\nu)}
&= \left.\begin{pmatrix}
0 & 0 & \lambda+1 
\\
0 & 2\lambda+4 & 0 
\\
\lambda+1 & 0 & 0 
\end{pmatrix}\right|_{\lambda=(\nu+2)+(n-2)}.
\end{align*}
Therefore, the generating function of the $3\times 3$ Gegenbauer polynomials reads
\begin{align*}
M(x;t)
&=\begin{pmatrix} 1 & 0 & 0 \\ 0 & 2 & 0 \\ 0 & 0 & 1 \end{pmatrix} \sum_{n=0}^\infty((\nu+2)+n)C_n^{(\nu+2)}(x)t^n
-\begin{pmatrix} 1 & 0 & 0 \\ 0 & 4 & 0 \\ 0 & 0 & 1 \end{pmatrix} \sum_{n=0}^\infty C_n^{(\nu+2)}(x)t^n
\\ &\quad
-2t\begin{pmatrix} 0 & 1 & 0 \\ 1 & 0 & 1 \\ 0 & 1 & 0 \end{pmatrix} \sum_{n=1}^\infty((\nu+2)+(n-1))C_{n-1}^{(\nu+2)}(x)t^{n-1}
\\ &\quad
+t^2\begin{pmatrix} 0 & 0 & 1 \\ 0 & 2 & 0 \\ 1 & 0 & 0 \end{pmatrix} \sum_{n=2}^\infty((\nu+2)+(n-2))C_{n-2}^{(\nu+2)}(x)t^{n-2}
\\ &\quad
+t^2\begin{pmatrix} 0 & 0 & 1 \\ 0 & 4 & 0 \\ 1 & 0 & 0 \end{pmatrix} \sum_{n=2}^\infty C_{n-2}^{(\nu+2)}(x)t^{n-2}
\displaybreak[2]\\
&=\begin{pmatrix} 1 & 0 & 0 \\ 0 & 2 & 0 \\ 0 & 0 & 1 \end{pmatrix} \frac{(\nu+2)(1-t^2)}{(1-2xt+t^2)^{\nu+3}}
-\begin{pmatrix} 1 & 0 & 0 \\ 0 & 4 & 0 \\ 0 & 0 & 1 \end{pmatrix} \frac1{(1-2xt+t^2)^{\nu+2}}
\\ &\quad
-\begin{pmatrix} 0 & 1 & 0 \\ 1 & 0 & 1 \\ 0 & 1 & 0 \end{pmatrix}
\frac{2(\nu+2)t(1-t^2)}{(1-2xt+t^2)^{\nu+3}}
\\ &\quad
+\begin{pmatrix} 0 & 0 & 1 \\ 0 & 2 & 0 \\ 1 & 0 & 0 \end{pmatrix} \frac{(\nu+2)t^2(1-t^2)}{(1-2xt+t^2)^{\nu+3}}
+\begin{pmatrix} 0 & 0 & 1 \\ 0 & 4 & 0 \\ 1 & 0 & 0 \end{pmatrix} \frac{t^2}{(1-2xt+t^2)^{\nu+2}},
\end{align*}
so that $M(x;t)\cdot(1-2xt+t^2)^{\nu+3}$ is a $3\times3$ matrix with entries from $\mathbb Z[\nu,x,t]$.

The computation shows that, for general $\ell>0$, the generating function $M(x;t)$ is a $(2\ell+\nobreak1)\times(2\ell+1)$ matrix multiple of the generating function \eqref{eq:gen-fun-scalar} with $\lambda=\nu+2\ell+\lfloor\ell\rfloor$, whose all entries are polynomials in $\nu$, $x$ and~$t$ with integer coefficients.


\section{Zeros of matrix-valued Gegenbauer polynomials}\label{sec:zerosentriesMVGegenbauerpols}

It is common to understand zeros of matrix-valued polynomials as those of their \emph{determinants}, see \cite{DamaPS} for further discussion and references. 
Such zeros serve as a natural analog of those for scalar polynomials: in the case of matrix-valued orthogonal polynomials for the measure supported on a real interval, their determinants have all zeros on the (internal part of the) interval (of multiplicities that do not exceed the size of the matrix); see \cite[Thm.~1.1]{DL-R96} and \cite[Cor.~4.4]{DL-RS99}.
For our matrix-valued Gegenbauer polynomials, the support of the measure is the interval $[-1,1]$, hence all the zeros of their determinants lie in the interior of this interval (and are symmetric with respect to the origin).
Theorem~\ref{thm:expansionFnumr} provides us with access to \emph{individual entries} of these matrix-valued polynomials, and we report on our findings in this section.
At the moment, these considerations do not seem to provide any particular insight on a connection between entry-wise zeros and the zeros of the determinant of the matrix-valued polynomial.
They may be part of a different analytic phenomenon that shows up in the matrix-valued setting. On the other hand, there is a clear interest in zeros of suitable 
linear combinations of scalar orthogonal polynomials, see e.g. Beardon and Driver \cite{BearD}, Dur\'an \cite{Dura25} 
and references given there. The linear combinations of Gegenbauer 
polynomials arising as entries of the matrix-valued Gegenbauer polynomials for which we study the structure of the zeros gives intriguing examples of such linear combinations, which are typically outside the classes studied in \cite{BearD}, \cite{Dura25}. 

For the clarity of exposition in this section, we introduce the notion of \emph{echelon} for entries of an $(2\ell+1)\times(2\ell+1)$-matrix $(a_{ij})_{0\le i,j\le2\ell}$. This corresponds to a `distance of the entry $a_{ij}$ to the boundary of the matrix', given explicitly by $\operatorname{ech}_\ell(i,j)=1+\min\{i,2\ell-\nobreak i,\allowbreak j,2\ell-j\}$. For example, entries located in the first or last rows, or in the first or last columns are referred to as being from the `first echelon': if $i\in\{0,2\ell\}$ or $j\in\{0,2\ell\}$ then $\operatorname{ech}_\ell(i,j)=1$.

The binomial factor
\[
\binom{2\ell}k\binom k{\frac12(k+i-j)}\binom{2\ell-k}{\frac12(i+j-k)}
\]
in the definition of $\ga_\ell(\nu;i,j,k)$ dictates the presence of at most $\operatorname{ech}_\ell(i,j)$ polynomials $C_m^{(\nu+2\ell)}(x)$ in the linear combination expressing the entry $\big(\wh P_n^{(\nu)}(x)\big)_{i,j}$.
In particular, the entries from the `first echelon' (when $\operatorname{ech}_\ell(i,j)=1$) are multiples of corresponding scalar Gegenbauer polynomials, so that their zeros are real and lie on the interval $(-1,1)$. They even inherit the zero interlacing property when we consider the same first-echelon entry of two consecutive matrix-valued Gegenbauer polynomials.

The zeros of the entries from the `second echelon' (when $\operatorname{ech}_\ell(i,j)=2$) are also real and belong to the measure support interval $[-1,1]$.
To see this, notice that these entries are always linear combinations of $C_m^{(\lambda)}(x)$ and $C_{m-2}^{(\lambda)}(x)$, where $\lambda=\nu+2\ell$, with coefficients \emph{of the same sign}.
On the other hand, the two scalar Gegenbauer polynomials in such combinations can be translated into \emph{consecutive} Jacobi polynomials%
\footnote{They are usually called $P_n^{(\alpha,\beta)}(x)$ but we try to avoid a conflicting notation with our matrix-valued Gegenbauer polynomials.}
$J_n^{(\alpha,\beta)}(x)$ with the help of expressions
\[
C_{2n}^{(\lambda)}(x)=\frac{(\lambda)_n}{(\frac12)_n}\,J_n^{(\lambda-\frac12,-\frac12)}(2x^2-1)
\quad\text{and}\quad
C_{2n+1}^{(\lambda)}(x)=\frac{(\lambda)_{n+1}}{(\frac12)_{n+1}}\,xJ_n^{(\lambda-\frac12,\frac12)}(2x^2-1).
\]
Finally, the zeros of two consecutive Jacobi polynomials (labelled by the same pair of parameters $(\alpha,\beta)$) lie on the interval $(-1,1)$ and interlace, so that the zeros of their linear combination with non-negative coefficients lie on the same interval by the Hermite--Kakeya--Obreschkoff theorem (see, for example, \cite[Thm.~6.3.8]{RahmS} or \cite[Prop.~2.10]{MFetal-24} for a more general statement when multiple zeros and non-strict interlacing are allowed).
This justifies the location of zeros of entries $\big(\wh P_n^{(\nu)}(x)\big)_{i,j}$ with $\operatorname{ech}_\ell(i,j)=2$.

\begin{remark}\label{rem:shift-k-real}
For the middle $(1,1)$-entry in the $3\times3$ case ($\ell=1$), Theorem~\ref{thm:expansionFnumr} gives the following explicit expression:
\[
\big(\wh P_n^{(\nu)}(x)\big)_{1,1}
=\frac{(\nu+n+2)\,n!\,\Gamma(\nu+2)}{2^{n-1}(\nu+n+1)^2\Gamma(\nu+n)}
\cdot\bigg(\frac{C_n^{(\nu+2)}(x)}{\nu+n+2}+\frac{C_{n-2}^{(\nu+2)}(x)}{\nu+n}\bigg),
\]
so that the polynomial is proportional to the sum $\wt C_n^{(\nu+2)}(x)+\wt C_{n-2}^{(\nu+2)}(x)$, where
$\wt C_n^{(\lambda)}(x)=C_n^{(\lambda)}(x)/(\lambda+n)$.
Though the argument from the last paragraph explains why the zeros of $\wt C_n^{(\lambda)}(x)+\wt C_{n-k}^{(\lambda)}(x)$ lie on the interval $-1<x<1$ for $k=1$ and~$2$, experimentally we have observed that this is also the case for other choices of shift $k\in\mathbb Z$.
\end{remark}

The first case when the `third echelon' shows up is the middle $(2,2)$-entry of the $5\times5$ Gegenbauer polynomials ($\ell=2$); we get
\begin{align*}
&
\big(\wh P_n^{(\nu)}(x)\big)_{2,2}
=\frac{3(\nu+n+4)\,n!\,\Gamma(\nu+4)}{2^{n-1}\Gamma(\nu+n)}
\cdot\bigg(\frac{C_n^{(\nu+4)}(x)}{(\nu+n+2)^2(\nu+n+3)^2(\nu+n+4)}
\\ &\quad
+\frac{4C_{n-2}^{(\nu+4)}(x)}{(\nu+n+1)^2(\nu+n+2)(\nu+n+3)^2}
+\frac{C_{n-4}^{(\nu+4)}(x)}{(\nu+n)(\nu+n+1)^2(\nu+n+2)^2}\bigg).
\end{align*}
The structure of zeros of these polynomials is somewhat peculiar: There is one zero $x=0$ for $n=1$; then a pair of purely imaginary (conjugate) zeros for $n=2$; then another pair and $x=0$ for $n=3$. Beyond this $n$ we start witnessing real zeros (always on the interval $-1<x<1$) and a pair of imaginary conjugates. The real zeros interlace when passing from $n-1$ to $n$, while the upper imaginary zero increase (up to a certain constant strictly less than~1) with~$n$.

\begin{figure}[t]
    \centering
    \begin{subfigure}[t]{0.32\textwidth}
        \centering
        \includegraphics[width=\textwidth]{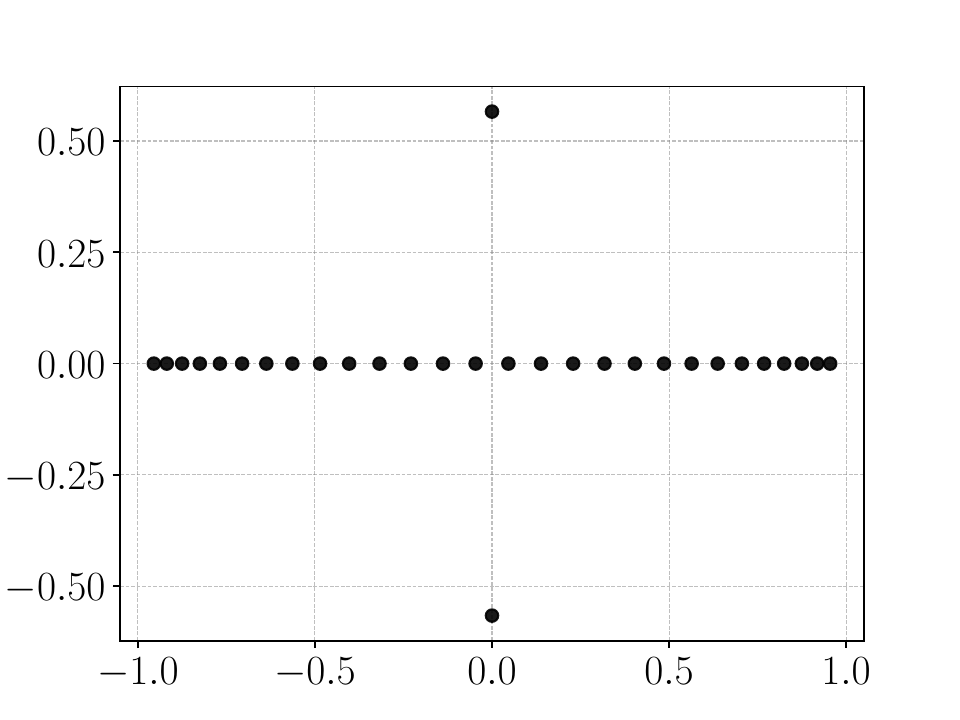}
    \end{subfigure}
    \hfill
    \begin{subfigure}[t]{0.32\textwidth}
        \centering
        \includegraphics[width=\textwidth]{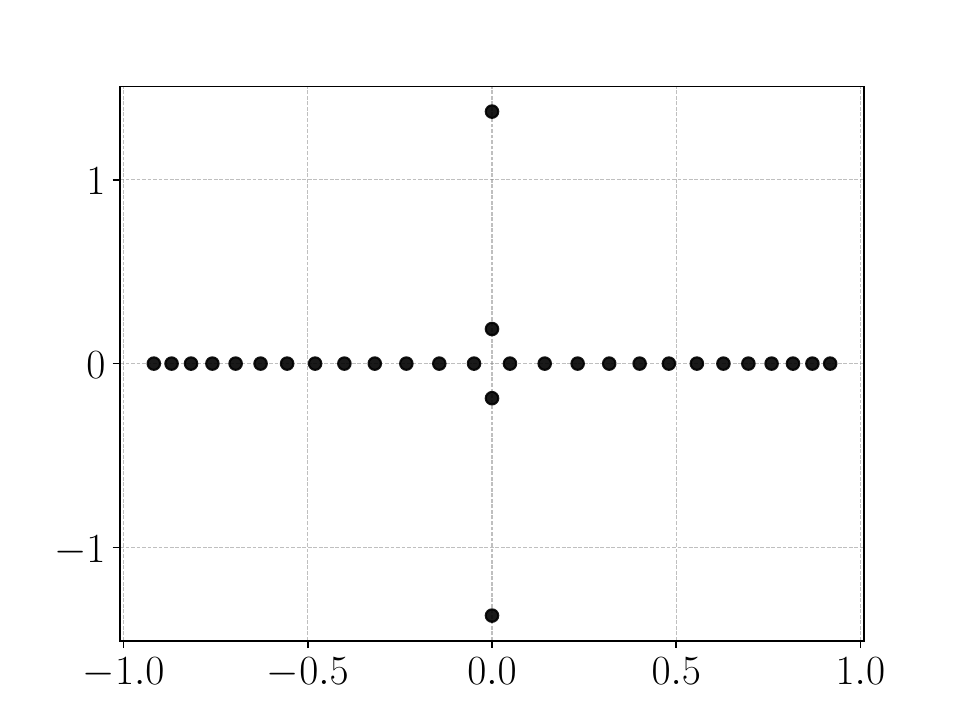}
    \end{subfigure}
    \begin{subfigure}[t]{0.32\textwidth}
        \centering
        \includegraphics[width=\textwidth]{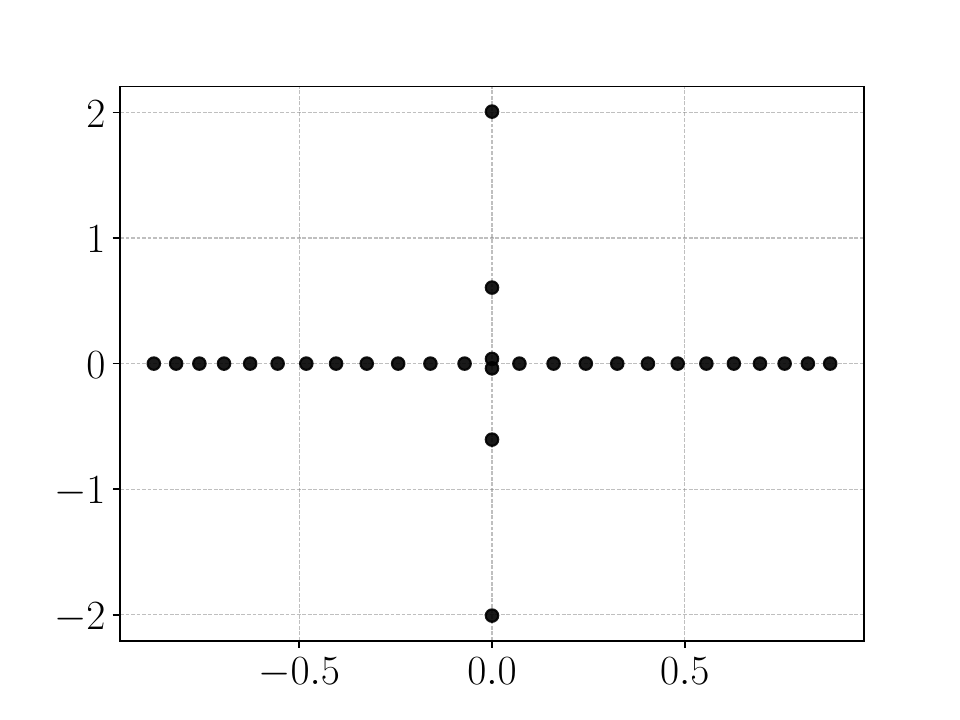}
    \end{subfigure}
    \caption{%
        Zeros of the middle entry $\big(\hat P^{(3)}_{30}(x)\big)_{\ell,\ell}$ for $\ell=2$ (left), $\ell=4$ (center) and $\ell=6$ (right)
    }
    \label{fig:EntryMiddle}
\end{figure}

Similar structures are present for `higher echelons' when more pairs of purely imaginary zeros come up, while all other zeros remain real\,---\,see Figure~\ref{fig:EntryMiddle} for selected instances.
So far we have not figured out reasons of this behavior.
Through a complicated route, our discussions of this phenomenon led us to consider general linear combinations of scalar orthogonal polynomials and finally resulted in a separate project\,---\,the details of this side investigation can be found in~\cite{KRZ24}.

We have not explored carefully the theme of zero loci of individual entries for other known families of \emph{symmetrizable} orthogonal polynomials. But the peculiar structure of zeros seems to persist, for instance, in `randomly chosen' examples of matrix-valued Hermite
polynomials from \cite{IKR19}; the fact that they can be normalized as symmetric is numerically supported in all these examples.
At the same time, we do not expect the zero location to be a `universal' phenomenon.
As has been pointed out to us by Arno Kuijlaars, the off-diagonal entries of matrix-valued orthogonal polynomials can have a very exotic distribution already in $2\times2$ situations; examples arising from a beautiful combinatorics of periodic hexagon tilings with period~2 are discussed in detail in \cite{GK21}\,---\,the off-diagonal zeros follow a quite unexpected pattern.


\section{Matrix matters and discussion}\label{sec:discussion}

The explicit connection formulas for the matrix-valued and scalar Gegenbauer polynomials in Theorems~\ref{thm:expansionGnumr} and~\ref{thm:expansionFnumr} raise a question about how general this phenomenon is for matrix-valued orthogonal polynomials.
They also suggest an approach to look for new families of matrix-valued polynomials by imposing the length in such expansions to depend only on the matrix size but not on the degree.
Decompositions of this type, in particular Theorem~\ref{thm:expansionFnumr}, hint at the possibility to investigate asymptotics of matrix-valued orthogonal polynomials using known ones for their scalar bases.
These formulas also offer numerous further directions for study and applications of matrix-valued polynomials.

The expressions in Theorem~\ref{thm:expansionFnumr} allow one to pursue analysis of a related Pad\'e problem for the matrix-valued generating function of the moments of the weight~\eqref{eq:defweight}.
The corresponding theoretical setup goes in complete parallel with the scalar version \cite{Duran96}.
This gives room to potential applications of these matrix-valued orthogonal polynomials in number theory, to arithmetic properties of the values of the generating function at rationals\,---\,the values that can be viewed as both matrices and entry-wise.
This arithmetic direction seems to be under-explored at the moment.

It would be also of great interest to understand the new differential-difference structure for the matrix-valued Gegenbauer polynomials given in Section~\ref{sec:differenceeqtFknnu} more conceptually.
We stress on the fact that it completely degenerates in the scalar case (for $1\times1$ matrices), so that it does not represent any classically familiar setting.

We are confident that the mathematics story of this note will continue in diverse\,---\,and quite remarkable!\,---\,directions.


\end{document}